\documentclass[12pt, oneside]{article}   	
\usepackage{geometry}                		
\usepackage{amsthm}
\usepackage{amsmath}
\usepackage{mathtools}
\usepackage{graphicx}				
							 	
\usepackage{amssymb}
\usepackage[all]{xy}
\usepackage{makeidx}
\usepackage{nameref}
\usepackage{xcolor}
\usepackage{tikz-cd}
\usetikzlibrary{arrows,decorations.pathreplacing,angles,quotes,bending}
\usepackage{mathscinet}
\usepackage{enumerate}
\usepackage{float}
\usepackage{enumitem}
\usepackage[font=scriptsize]{caption}
\makeindex
\theoremstyle{plain}
\makeatletter
\newcommand{\newreptheorem}[2]{\newtheorem*{rep@#1}{\rep@title}\newenvironment{rep#1}[1]{\def\rep@title{#2 \ref*{##1}}\begin{rep@#1}}{\end{rep@#1}}}
\makeatother

\newtheorem{theorem}{Theorem}[section]
\newreptheorem{theorem}{Theorem}

\newtheorem{lemma}[theorem]{Lemma}
\newtheorem{corollary}[theorem]{Corollary}
\newtheorem{proposition}[theorem]{Proposition}
\theoremstyle{definition}
\newtheorem{definition}[theorem]{Definition}

\title{A novel connection between integral binary quadratic forms and knot polynomials}
\author{Amitesh Datta}
\begin{document}
\maketitle

\abstract{We establish a novel connection between algebraic number theory and knot theory. We show that the number of equivalence classes of integral binary quadratic forms of discriminant $t^2 - 4$ (for $t\neq \pm 2$) is equal to the number of isotopy classes of links in $\mathbb{S}^3$ with prescribed values (depending on $t$) of three classical link invariants. The equality arises from a natural algebraic correspondence between integral binary quadratic forms (of discriminant $t^2 - 4$ for $t\neq \pm 2$) and isotopy classes of links of braid index at most three. In particular, the class numbers of certain quadratic number fields precisely measure the failure of the Alexander/Jones polynomial to distinguish non-isotopic links of braid index at most three.}

\tableofcontents
\section{Introduction}
In this paper, we use the algebraic structure of the braid group $B_3$ to establish a fundamental connection between the algebraic number theory of quadratic number fields and knot theory in $\mathbb{S}^3$. Firstly, we develop a natural algebraic correspondence between integral binary quadratic forms (of discriminant $t^2 - 4$ for $t\neq \pm 2$) and isotopy classes of links of braid index at most three. Secondly, we use this correspondence to show that the number of integral binary quadratic forms of discriminant $t^2 - 4$ (for some integer $t\neq \pm 2$) is equal to the number of isotopy classes of links in $\mathbb{S}^3$ with prescribed values (depending on $t$) of three classical link invariants (the Alexander/Jones polynomial, the braid index, and the writhe of certain link diagrams).

The connection developed in this paper is a logical bridge between number theory and topology. Firstly, results in number theory can be established using topology (e.g., one can establish lower bounds on the sizes of ideal class groups of certain quadratic number fields by producing non-isotopic links with the same Alexander/Jones polynomial). Secondly, results in topology can be established using number theory (e.g., one can produce non-isotopic links by producing inequivalent ideal classes in a quadratic number field).

\subsection{Summary of main results}

We briefly recall terminology before stating the main result. Let $L\subseteq \mathbb{S}^3$ be a link. The Alexander polynomial of $L$ is denoted by $\Delta_{L}\left(q\right)\in \mathbb{Z}\left[\sqrt{q},\sqrt{q^{-1}}\right]$ and the Jones polynomial of $L$ is denoted by $V_{L}\left(q\right)\in \mathbb{Z}\left[\sqrt{q},\sqrt{q^{-1}}\right]$. We use the definitions in~\cite{jones1987hecke}.

The \textit{braid index $b\left(L\right)$ of a link $L$} is the minimal positive integer $m$ such that $L$ is isotopic to the closure of a braid in the braid group $B_m$ on $m$ strands. The braid index of a link is a link invariant and Yamada~\cite{yamadaseifertcircles} proved that the braid index of $L$ is equal to the minimal number of Seifert circles in a link diagram for $L$. 

The \textit{writhe $e\left(D\right)$ of an oriented link diagram $D$} is the signed number of crossings in $D$. The writhe is not a link invariant. However, if $L$ is a link of braid index $m\leq 3$, then the writhe of a diagram for $L$ arising as the closure of a braid in $B_m$ is a link invariant to which we refer as $e\left(L\right)$. (Jones~\cite{jones1987hecke} proposes the possibility that the previous statement is true without restriction on $m$ but this more general statement remains open.) 

The main result of the paper is Theorem~\ref{tmain}, which is a little notationally involved. The following statement is a special case of Theorem~\ref{tmain}, and we state it here for ease of exposition.

\begin{theorem}
\label{tmainconsequence}
Let $t$ and $n$ be integers with $t\neq \pm 2$. Let $h\left(t\right)$ denote the number of equivalence classes of integral binary quadratic forms of discriminant $t^2 - 4$. If $L$ is a link, then we fix $P_L$ to be either the Alexander polynomial of $L$ or the Jones polynomial of $L$. Let $p_{t,n}$ be the number of isotopy classes of links of braid index three such that:
\begin{description}
\item[(i)] We have $e\left(L\right) = n$.
\item[(ii)] The special value $P_L\left(-1\right) = i^{n}\left(t-2\right)$, where $i= \sqrt{-1}$ is the imaginary unit.
\end{description}
If $n< -\left|t-3\right|-11$, then $h\left(t\right) = \sum_{j=0}^{11} p_{t,n+j}$. In general, $h\left(t\right)\geq \sum_{j=0}^{11} p_{t,n+j}$ for all $t$ and $n$.
\end{theorem}

If $L$ is a link of braid index three, then the writhe of $L$ is equal to $n$ and the special value of both the Alexander polynomial and the Jones polynomial at $-1$ is equal to $i^{n}\left(t-2\right)$ for a unique pair of integers $t$ and $n$. In particular, the isotopy class of every link of braid index three contributes to $p_{t,n}$ for a unique pair of integers $t$ and $n$.

Lagrange's theorem on the finiteness of the number of equivalence classes of integral binary quadratic forms of a fixed discriminant is equivalent to the following statement, by means of our main result (Theorem~\ref{tmain}).

\begin{theorem}
\label{tfiniteness}
Let $P\in \mathbb{Z}\left[\sqrt{q},\sqrt{q}^{-1}\right]$. The number of isotopy classes of links in $\mathbb{S}^3$ of braid index at most three with either Alexander polynomial equal to $P$ or Jones polynomial equal to $P$ is finite.
\end{theorem}

We also prove the following remarkable symmetry and periodicity of counts of isotopy classes of links. 

\begin{theorem}
\label{tsymmetry}
Let $p_{t,n}$ be as in the statement of Theorem~\ref{tmainconsequence}. If $n<-\left|t + 3\right|-17$, then $\sum_{j=0}^{11} p_{t,n+j} = \sum_{j=0}^{11} p_{-t,n+6+j}$. If $n<-\left|t-3\right| - 12$, then $p_{t,n} = p_{t,n+12}$.
\end{theorem}

Let us elaborate on the correspondence between isotopy classes of links and equivalence classes of integral binary quadratic forms that leads to these results. If $t\neq \pm 2$, then we will show that every isotopy class of links $L$ contributing to the count $p_{t,n}$ corresponds to a set $S_L$ of equivalence classes of integral binary quadratic forms of discriminant $t^2-4$, where the cardinality of $S_L$ is at most two. Furthermore, we will show that every equivalence class of integral binary quadratic forms of discriminant $t^2 - 4$ arises as an element of $S_L$ for infinitely many distinct isotopy classes of links $L$ (corresponding to the infinite number of possibilities of the integer $n$ in the statement of Theorem~\ref{tmainconsequence}). Finally, if $L$ and $L'$ are distinct isotopy classes of links with $\left|e\left(L\right) - e\left(L'\right)\right|<12$, then we will show that $S_L\cap S_{L'} = \emptyset$. 

We note the following application of this correspondence that allows one to use topological methods to establish results in number theory, and to use number-theoretic methods to establish results in topology: If one uses finer topological invariants of links compared to the Alexander and Jones polynomial, such as the knot group, knot Floer homology, Khovanov homology etc., then one can use Theorem~\ref{tmainconsequence} to distinguish inequivalent binary quadratic forms of discriminant $t^2 - 4$. Conversely, number-theoretic methods to distinguish inequivalent binary quadratic forms of discriminant $t^2 - 4$ also distinguish non-isotopic links.

\subsection{Outline of the proof}

We outline the proof of Theorem~\ref{tmainconsequence} and simultaneously the organization of the paper. In Section~\ref{sconjugationB3SL2}, we study the short exact sequence \[1\to \left\langle \Delta^4 \right\rangle \to B_3\to \text{SL}_2\left(\mathbb{Z}\right)\to 1,\] where $B_3$ is the braid group and $\Delta^2$ is a generator of the infinite cyclic center of $B_3$. We use the short exact sequence to establish a concrete connection between conjugacy classes in $B_3$ and conjugacy classes in $\text{SL}_2\left(\mathbb{Z}\right)$. In Section~\ref{sLatimerMacDuffee}, we recall that conjugacy classes in $\text{SL}_2\left(\mathbb{Z}\right)$ of trace $t\neq \pm 2$ bijectively correspond to equivalence classes of integral binary quadratic forms with discriminant $t^2 - 4$ (Chowla-Cowles-Cowles~\cite{ChowlaconjugacyclassesSL2}). In Section~\ref{sBirmanMenasco}, we study the natural correspondence between conjugacy classes in $B_3$ and isotopy classes of links of braid index at most three in $\mathbb{S}^3$ via the braid closure operation. We use a precise characterization of this correspondence, and the results in Section~\ref{sconjugationB3SL2}, to establish a fundamental connection between isotopy classes of links of braid index three in $\mathbb{S}^3$ and equivalence classes of integral binary quadratic forms with discriminant $t^2 - 4$ (for $t\neq \pm 2$). 

In Section~\ref{sAlexanderJones}, we discuss the connection between the Alexander/Jones polynomial and the map $B_3\to \text{SL}_2\left(\mathbb{Z}\right)$. Let $L$ be obtained as the braid closure $\overline{g}$ for $g\in B_3$. Let $\beta_3:B_3\to \text{GL}_2\left(\mathbb{Z}\left[q^{\pm 1}\right]\right)$ be the (reduced) Burau representation of the braid group $B_3$, and let $\epsilon:B_3\to \mathbb{Z}$ be the abelianization homomorphism. The Alexander polynomial and the Jones polynomial of $L$ can be derived from the trace of $\beta_3\left(g\right)$ and the exponent sum $\epsilon\left(g\right)$ if $g$ is expressed as a word in the Artin generators. We observe that the special value of $\beta_3$ at $q = -1$ is the map $B_3\to \text{SL}_2\left(\mathbb{Z}\right)$ in the aforementioned short exact sequence, and if $L$ has braid index three, then $\epsilon\left(g\right)$ is equal to the writhe $e\left(L\right)$. Finally, in Section~\ref{sproof} we state and prove Theorem~\ref{tmain}, which is the main result of this paper (of which Theorem~\ref{tmainconsequence} is a special case, and Theorem~\ref{tfiniteness} is a corollary).

The author is presently extending the connection established in the paper to a connection between class groups of higher degree number fields and isotopy classes of arbitrary links in $\mathbb{S}^3$. The extension relies on a deeper study of conjugacy classes in the braid groups and linear representations of the braid groups. 

\section{The braid group $B_3$ and the group $\text{SL}_2\left(\mathbb{Z}\right)$}
\label{sconjugationB3SL2}

In this section, we establish a connection between conjugacy classes in the braid group $B_3$ and conjugacy classes in the group $\text{SL}_2\left(\mathbb{Z}\right)$. The former is connected to isotopy classes of links in $\mathbb{S}^3$ of braid index at most three (via the braid closure operation) and the latter is connected to equivalence classes of integral binary quadratic forms. In particular, the connection established in this section will be the formal basis for the main result of this paper. 

If $G$ is a group, then we use the notation $C\left(G\right)$ to denote the set of conjugacy classes in $G$. The following statement is straightforward.

\begin{proposition}
\label{pformalconjugacy}
If $G\to H$ is a group homomorphism, then the image of a conjugacy class in $G$ is contained in a conjugacy class in $H$. If $G\to H$ is a surjective group homomorphism, then the image of a conjugacy class in $G$ is equal to a conjugacy class in $H$.
\end{proposition}

Proposition~\ref{pformalconjugacy} implies that if $G\to H$ is a surjective group homomorphism, then there is an induced surjective map $C\left(G\right)\to C\left(H\right)$ of sets. We recall that \[B_3 = \left\langle \sigma_1,\sigma_2 : \sigma_1\sigma_2\sigma_1 = \sigma_2\sigma_1\sigma_2 \right\rangle\] is the Artin presentation of the braid group $B_3$, and $\Delta = \sigma_1\sigma_2\sigma_1 = \sigma_2\sigma_1\sigma_2$ is the \textit{Garside element} of $B_3$. The element $\Delta^2$ generates the center of $B_3$. We also recall that \[\text{SL}_2\left(\mathbb{Z}\right) = \left\{\begin{bmatrix} a & b \\ c & d \end{bmatrix}: a,b,c,d\in \mathbb{Z}\text{ and } ad-bc = 1\right\}\] is the group of $2\times 2$ integer matrices with determinant $+1$.  Let \[S = \begin{bmatrix} 1 & 1 \\ 0 & 1 \end{bmatrix},\text{ } T = \begin{bmatrix} 1 & 0 \\ -1 & 1 \end{bmatrix}.\] The following statement is classical.

\begin{proposition}
\label{propfpPSL_2}
We have the following finite presentation \[\text{SL}_2\left(\mathbb{Z}\right) = \left\langle S,T : STS = TST,\text{ }\left(ST\right)^6 = I \right\rangle.\]
\end{proposition}

Let us define $\phi:B_3\to \text{SL}_2\left(\mathbb{Z}\right)$ in terms of generators in the Artin presentation of $B_3$ by the rules $\phi\left(\sigma_1\right) = S$ and $\phi\left(\sigma_2\right) = T$. Proposition~\ref{propfpPSL_2} implies that $\phi$ is a well-defined surjective homomorphism.

\begin{proposition}
\label{propB_3PSL_2}
The sequence of maps \[1\to \left\langle \Delta^4 \right\rangle \to B_3\to \text{SL}_2\left(\mathbb{Z}\right)\to 1\] is a short exact sequence.
\end{proposition}
\begin{proof}
The statement is a consequence of Proposition~\ref{propfpPSL_2}, since $\phi\left(\Delta\right) = STS$. Indeed, $\Delta^4$ is central in $B_3$, and thus the subgroup of $B_3$ normally generated by $\Delta^4$ equals the subgroup of $B_3$ generated by $\Delta^4$.
\end{proof}

We refer to the short exact sequence in Proposition~\ref{propB_3PSL_2} as the \textit{fundamental short exact sequence}. 

\begin{proposition}
\label{pconjugacyconnection}
If $g,h\in B_3$, then the elements $\phi\left(g\right),\phi\left(h\right)\in \text{SL}_2\left(\mathbb{Z}\right)$ are conjugate in $\text{SL}_2\left(\mathbb{Z}\right)$ if and only if the elements $\Delta^{4k}g,h\in B_3$ are conjugate in $B_3$ for some $k\in \mathbb{Z}$. 
\end{proposition}
\begin{proof}
The statement is an immediate consequence of the fundamental short exact sequence, since $\Delta^4$ is central in $B_3$.
\end{proof}

Let $\epsilon:B_3\to \mathbb{Z}$ denote the abelianization homomorphism, which is uniquely defined by the rule $\epsilon\left(\sigma_1\right) = 1 = \epsilon\left(\sigma_2\right)$. If $g\in B_3$ is represented as a word in the generators $\sigma_1$ and $\sigma_2$, then $\epsilon\left(g\right)$ is the exponent sum of $g$. In particular, $\epsilon:B_3\to \mathbb{Z}$ is a class function on $B_3$. We observe that $\epsilon\left(\Delta\right) = 3$ and we record the following consequence of Proposition~\ref{pconjugacyconnection}.

\begin{proposition}
\label{pepsilonconjugacyconnection}
The elements $g,h\in B_3$ are conjugate in $B_3$ if and only if $\phi\left(g\right),\phi\left(h\right)\in \text{SL}_2\left(\mathbb{Z}\right)$ are conjugate in $\text{SL}_2\left(\mathbb{Z}\right)$ and $\epsilon\left(g\right) = \epsilon\left(h\right)$. 
\end{proposition}
\begin{proof}
Let us assume that $\Delta^{4k}g,h\in B_3$ are conjugate in $B_3$ for some $k\in \mathbb{Z}$. In this case, we have $\epsilon\left(\Delta^{4k}g\right) = \epsilon\left(h\right)$. In particular, $\epsilon\left(g\right) = \epsilon\left(h\right)$ if and only if $k=0$ and $g,h\in B_3$ are conjugate in $B_3$. The statement is now a consequence of Proposition~\ref{pconjugacyconnection}.
\end{proof}

Proposition~\ref{pepsilonconjugacyconnection} implies that the fibers of the induced surjective map $\phi_{\ast}:C\left(B_3\right)\to C\left(\text{SL}_2\left(\mathbb{Z}\right)\right)$ are parametrized by $\mathbb{Z}$. In fact, if $c\in C\left(B_3\right)$, and if $\Delta^{4k}c\in C\left(B_3\right)$ is defined to be $\{\Delta^{4k}g:g\in c\}$, then Proposition~\ref{pconjugacyconnection} implies that $\{\Delta^{4k}c\}_{k\in \mathbb{Z}}$ is the fiber of $\phi_{\ast}:C\left(B_3\right)\to C\left(\text{SL}_2\left(\mathbb{Z}\right)\right)$ containing $c$. 

The \textit{trace map} $\text{tr}:\text{SL}_2\left(\mathbb{Z}\right)\to \mathbb{Z}$ lifts to a \textit{trace map} $\text{tr}:B_3\to \mathbb{Z}$ via the map $\phi:B_3\to \text{SL}_2\left(\mathbb{Z}\right)$. The trace maps on $\text{SL}_2\left(\mathbb{Z}\right)$ and $B_3$ are class functions on $\text{SL}_2\left(\mathbb{Z}\right)$ and $B_3$. In particular, \textit{the trace of a conjugacy class} in either group is well-defined. The \textit{exponent of a conjugacy class} in $B_3$ is defined to be the exponent sum of any element of the conjugacy class, and it is also well-defined on $B_3$ (but only well-defined modulo $12$ on $\text{SL}_2\left(\mathbb{Z}\right)$).

\begin{definition}
If $t$ and $n$ are integers, then $X_{t,n}$ is the set of conjugacy classes in $B_3$ of trace $t$ and exponent $n$ and $Y_t$ is the set of conjugacy classes in $\text{SL}_2\left(\mathbb{Z}\right)$ of trace $t$.
\end{definition}

The main result of this section is the following set of statements concerning the sets $X_{t,n}$ and $Y_{t}$, as well as the relationship between them via the map $\phi:B_3\to \text{SL}_2\left(\mathbb{Z}\right)$.

\begin{lemma}
\label{lmainconnection}
\begin{description}
\item[(i)] The restriction of $\phi_{\ast}:X_{t,n}\to Y_t$ is injective. 
\item[(ii)] We have compatible bijections $X_{t,n}\cong X_{-t,n+6}$ and $Y_t\cong Y_{-t}$ with respect to $\phi_{\ast}$.
\item[(iii)] We have a bijection $X_{t,n}\cong X_{t,n+12}$ compatible with respect to $\phi_{\ast}$.
\item[(iv)] The restriction $\phi_{\ast}:\coprod_{j=0}^{11} X_{t,n+j}\to Y_t$ is a bijection.
\end{description}
\end{lemma}
\begin{proof}
\begin{description}
\item[(i)] The statement is a consequence of Proposition~\ref{pepsilonconjugacyconnection}.
\item[(ii)] We have $\phi\left(\Delta^2\right) = \begin{bmatrix} -1 & 0 \\ 0 & -1 \end{bmatrix}$. In particular, the map $B_3\to B_3$ of sets defined by multiplication by $\Delta^2$ induces a bijection $X_{t,n}\cong X_{-t,n+6}$ and the map $\text{SL}_2\left(\mathbb{Z}\right)\to \text{SL}_2\left(\mathbb{Z}\right)$ of sets defined by multiplication by $\phi\left(\Delta^2\right)$ induces a bijection $Y_t\cong Y_{-t}$.
\item[(iii)] The element $\Delta^4$ is in the kernel of $\phi$, and the map $B_3\to B_3$ of sets defined by multiplication by $\Delta^4$ induces a bijection $X_{t,n}\cong X_{t,n+12}$.
\item[(iv)] The statement is a consequence of the surjectivity of $\phi_{\ast}:C\left(B_3\right)\to C\left(\text{SL}_2\left(\mathbb{Z}\right)\right)$ and Proposition~\ref{pconjugacyconnection}.
\end{description}
\end{proof}

\section{The group $\text{SL}_2\left(\mathbb{Z}\right)$ and integral binary quadratic forms}
\label{sLatimerMacDuffee}
In this section, we recall the bijective correspondence between conjugacy classes in $\text{SL}_2\left(\mathbb{Z}\right)$ of trace $t$ (the set $Y_t$) and equivalence classes of integral binary quadratic forms of discriminant $t^2 - 4$, if $t\neq \pm 2$. 

An \textit{integral binary quadratic form} is an expression $f\left(x,y\right) = ax^2 + bxy + cy^2$ where $a,b,c\in \mathbb{Z}$. The group $\text{SL}_2\left(\mathbb{Z}\right)$ acts on the set of integral binary quadratic forms by the rule \[\left(\begin{bmatrix} \alpha & \beta \\ \gamma & \delta \end{bmatrix} \cdot f\right)\left(x,y\right) = f\left(\alpha x + \beta y, \gamma x + \delta y\right),\] if $\begin{bmatrix} \alpha & \beta \\ \gamma & \delta \end{bmatrix}\in \text{SL}_2\left(\mathbb{Z}\right)$ and $f\left(x,y\right)$ is an integral binary quadratic form. A pair of integral binary quadratic forms are \textit{equivalent} if they are in the same orbit with respect to this action. The \textit{discriminant of an integral binary quadratic form} $ax^2+bxy+cy^2$ is $b^2-4ac\in \mathbb{Z}$, and it is an invariant of the equivalence class of $ax^2+bxy+cy^2$.

The following result is due to Chowla-Cowles-Cowles~\cite{ChowlaconjugacyclassesSL2}, and it is closely related to a special case of the Latimer-MacDuffee theorem. 

\begin{theorem}[Chowla-Cowles-Cowles]
If $t\neq \pm 2$, then the correspondence \[\begin{bmatrix} a & b \\ c & d \end{bmatrix}\to bx^2 + \left(d-a\right)xy - cy^2\] defines a bijection from the set of conjugacy classes in $\text{SL}_2\left(\mathbb{Z}\right)$ of trace $t$ to the set of equivalence classes of integral binary quadratic forms of discriminant $t^2 - 4$. 
\end{theorem}

Let $h\left(t\right)$ be the number of equivalence classes of integral binary quadratic forms of discriminant $t^2 - 4$. The Chowla-Cowles-Cowles theorem implies that $h\left(t\right) = \left|Y_t\right|$ if $t\neq \pm 2$.

\begin{lemma}
\label{lLatimer-MacDuffee}
If $t\neq \pm 2$ and $n$ are integers, then $h\left(t\right) = \left|\coprod_{j=0}^{11} X_{t,n+j}\right|$.
\end{lemma}
\begin{proof}
The statement is a consequence of Lemma~\ref{lmainconnection}~\textbf{(iv)} and the Chowla-Cowles-Cowles theorem.
\end{proof} 

In the rest of the paper, we will interpret $\left|\coprod_{j=0}^{11} X_{t,n+j}\right|$ as a count of isotopy classes of links satisfying the conditions in the statement of Theorem~\ref{tmainconsequence} in the Introduction.

\section{The Birman-Menasco classification}
\label{sBirmanMenasco}
Let ${\cal L}_3$ denote the set of isotopy classes of links with braid index at most three. We recall that $C\left(B_3\right)$ denotes the set of conjugacy classes in the braid group $B_3$. We have a map $C\left(B_3\right)\to {\cal L}_3$ induced by the braid closure operation. 

In~\cite{birmanmenascoclosed3braids}, Birman-Menasco determined the fibers of the map $C\left(B_3\right)\to {\cal L}_3$. We will use this result to relate counts of conjugacy classes in the braid group $B_3$ to counts of isotopy classes of links of braid index at most three in $\mathbb{S}^3$. If $g\in B_3$, then we denote the conjugacy class of $g$ by $C\left(g\right)$.

\begin{theorem}[Birman-Menasco]
Let $L\in {\cal L}_3$. The fiber of the map $C\left(B_3\right)\to {\cal L}_3$ over $L$ is a single element except in the following cases:
\begin{description}
\item[(i)] If $L$ is the unknot, then the fiber of $L$ consists of the conjugacy classes $C\left(\sigma_1\sigma_2\right)$, $C\left(\sigma_1\sigma_2^{-1}\right)$, and $C\left(\sigma_1^{-1}\sigma_2^{-1}\right)$.
\item[(ii)] If $L$ is a $\left(2,k\right)$ torus link for $k\neq \pm 1$, then the fiber of $L$ consists of the conjugacy classes $C\left(\sigma_1^{k}\sigma_2\right)$ and $C\left(\sigma_1^{k}\sigma_2^{-1}\right)$. 
\item[(iii)] If $L$ is the braid closure of either \[\Delta^{2k}\sigma_1^{-1}\sigma_2^u\sigma_1^{-v}\sigma_2^{w} \text{ or } \Delta^{2k}\sigma_1^{-1}\sigma_2^{w}\sigma_1^{-v}\sigma_2^{u},\] where $k\in \{0,1\}$ and $\left(u,v,w\right)\in \mathbb{N}^3$ is a triple of positive integers such that $u\neq w$ and $v\geq 2$, then the fiber of $L$ consists of the two distinct conjugacy classes of these braids. 
\item[(iv)] If $L$ is the braid closure of either \[\Delta^{2k}\sigma_1^{-1}\sigma_2^{u}\sigma_1^{-1}\sigma_2^{v}\sigma_1^{-1}\sigma_2^{w} \text{ or } \Delta^{2k}\sigma_1^{-1}\sigma_2^{u}\sigma_1^{-1}\sigma_2^{w}\sigma_1^{-1}\sigma_2^{v},\] where $k\in \{1,2\}$ and $\left(u,v,w\right)\in \mathbb{N}^3$ is a triple of distinct positive integers, then the fiber of $L$ consists of the two distinct conjugacy classes of these braids.
\end{description}
\end{theorem}

We compute the traces and exponents of the conjugacy classes in fibers of $C\left(B_3\right)\to {\cal L}_3$ with cardinality greater than one in the following statement.

\begin{proposition}
\label{ptraceexponent}
\begin{description}
\item[(i)] The trace and exponent of the conjugacy class $C\left(\sigma_1\sigma_2\right)$ are $\text{tr} = 1$ and $\epsilon = 2$, respectively. The trace and exponent of the conjugacy class $C\left(\sigma_1\sigma_2^{-1}\right)$ are $\text{tr} = 3$ and $\epsilon = 0$. The trace and exponent of the conjugacy class $C\left(\sigma_1^{-1}\sigma_2^{-1}\right)$ are $\text{tr} = 1$ and $\epsilon = -2$, respectively.
\item[(ii)] The trace and exponent of the conjugacy class $C\left(\sigma_1^{k}\sigma_2\right)$ are $\text{tr} = 2 - k$ and $\epsilon = k+1$. The trace and exponent of the conjugacy class $C\left(\sigma_1^{k}\sigma_2^{-1}\right)$ are $\text{tr} = 2 + k$ and $\epsilon = k-1$. 
\item[(iii)] If $k\in \{0,1\}$ and if $\left(u,v,w\right)\in \mathbb{N}^3$ is a triple of positive integers such that $u\neq w$ and $v\geq 2$, then the trace and exponent of the distinct conjugacy classes $C\left(\Delta^{2k}\sigma_1^{-1}\sigma_2^u\sigma_1^{-v}\sigma_2^{w}\right)$ and $C\left(\Delta^{2k}\sigma_1^{-1}\sigma_2^{w}\sigma_1^{-v}\sigma_2^{u}\right)$ are \begin{align*} \text{tr} &= \left(-1\right)^{k}\left(2+\left(u+w\right)\left(1+v\right)+uvw\right) \\  \epsilon &= u+w-v-1+6k,\end{align*} respectively. 
\item[(iv)] If $k\in \{1,2\}$ and if $\left(u,v,w\right)\in \mathbb{N}^3$ is a triple of distinct positive integers, then the trace and exponent of the distinct conjugacy classes $C\left(\Delta^{2k}\sigma_1^{-1}\sigma_2^u\sigma_1^{-1}\sigma_2^{v}\sigma_1^{-1}\sigma_2^{w}\right)$ and $C\left(\Delta^{2k}\sigma_1^{-1}\sigma_2^{u}\sigma_1^{-1}\sigma_2^{w}\sigma_1^{-1}\sigma_2^{v}\right)$ are \begin{align*}\text{tr} &= \left(-1\right)^{k}(1+u+v+w+u\left(1+v\right)+v\left(1+w\right)+w\left(1+u\right) \\ &+\left(1+u\right)\left(1+v\right)\left(1+w\right)) \\ \epsilon &= u+v+w-3+6k,\end{align*} respectively.
\end{description}
\end{proposition}

If $L$ is a link of braid index $m$, then the writhe of a diagram for $L$ arising as the closure of a braid $g\in B_m$ is $\epsilon\left(g\right)$, where $\epsilon:B_m\to \mathbb{Z}$ is the abelianization homomorphism. We derive the following statement using Proposition~\ref{ptraceexponent} and the work of Birman-Menasco.

\begin{corollary}
\label{cwritheinvariant}
If $L$ is a link of braid index $m\leq 3$, then the writhe of a diagram for $L$ arising as the braid closure of a braid in $B_m$ is a link invariant (i.e., it is independent of the choice of such diagram).
\end{corollary}
\begin{proof}
The only link of braid index one is the unknot, and the statement is straightforward in this case. The only links of braid index two are the $\left(2,k\right)$ torus links for $k\neq \pm 1$, and the $\left(2,k\right)$ torus link arises as the braid closure of $\sigma_1^k$ where $B_2 = \left\langle \sigma_1 \right\rangle$ is the braid group on two strands. Furthermore, the $\left(2,k\right)$ torus link is not isotopic to the $\left(2,k'\right)$ torus link for $k\neq k'$. In particular, if $L$ is a $\left(2,k\right)$ torus link for $k\neq \pm 1$, then the writhe of $L$ is equal to $k$. Finally, if $L$ has braid index three, then the statement is a consequence of the Birman-Menasco theorem and Proposition~\ref{ptraceexponent}.
\end{proof}

Let $i = \sqrt{-1}$ be the imaginary unit. The Birman-Menasco theorem implies that we have a well-defined invariant of links ${\cal L}_3\to \mathbb{C}$ defined by $L\to i^{\epsilon\left(g\right)}\left(\text{tr}\left(g\right) - 2\right)$ if $L = \overline{g}\in {\cal L}_3$ for $g\in B_3$. In fact, we will see that this invariant is the special value of either the Alexander polynomial or the Jones polynomial of $L$ at $-1$ in Section~\ref{sAlexanderJones}, and we will suggestively denote it by $P_L\left(-1\right)$.

The next step is to use the Birman-Menasco theorem to explicitly determine the relationship between $X_{t,n}$ and the set of isotopy classes of links in $\mathbb{S}^3$ satisfying the conditions in the statement of Theorem~\ref{tmainconsequence} in the Introduction.

\begin{definition}
\label{dMnumbers}
Let $X'_{t,n}$ be the set of isotopy classes of links of braid index three with $P_L\left(-1\right) = i^n\left(t-2\right)$ and writhe $e\left(L\right) = n$. Let $M_{t,n}'$ be the cardinality of the union of the following two sets:
\begin{align*} \{\left(u,v,w,k\right)\in \mathbb{N}^3\times \{0,1\}:&u\neq w\text{ and }v\geq 2 \\ &n = u+w-v-1+6k\text{ and } \\ &\left(-1\right)^k t = 2+\left(u+w\right)\left(1+v\right)+uvw\} \\
\{\left(u,v,w,k\right)\in \mathbb{N}^3\times \{1,2\}:&u,v,w\text{ are distinct} \\ &n = u + v + w - 3 + 6k\text{ and } \\ &\left(-1\right)^k t = 1+ u + v+ w \\ &+u\left(1+v\right)+ v\left(1+w\right) + w\left(1+u\right) \\&+ \left(1+u\right)\left(1+v\right)\left(1+w\right)\}. \end{align*}
If $t\neq 1,3$ and $n\in \{\pm \left(t-3\right)\}$, then define $M_{t,n} = M_{t,n}' + 1$. If $t = 1$ and $n = \pm 2$, or if $t = 3$ and $n = 0$, then define $M_{t,n} = M_{t,n}' + 1$. In all other cases, define $M_{t,n} = M_{t,n}'$. 
\end{definition}

The definition of $M_{t,n}$ (Definition~\ref{dMnumbers}) is designed precisely so that the following statement is true. 

\begin{lemma}
\label{lbraidlinkdiscrepancy}
If $t$ and $n$ are integers, then $\left|X_{t,n}\right| = \left|X_{t,n}'\right| + M_{t,n}$.
\end{lemma}
\begin{proof}
The statement is a consequence of the Birman-Menasco theorem and Proposition~\ref{ptraceexponent}. Let $X_{t,n}^{\circ}\subseteq X_{t,n}$ be the subset of $X_{t,n}$ consisting of all conjugacy classes such that their braid closure is a link of braid index three. The braid closure operation induces a surjective map $X_{t,n}^{\circ}\to X_{t,n}'$, where the fibers have cardinality one or two. Furthermore, the number of fibers of cardinality two is equal to $M_{t,n}'$. We deduce that $\left|X_{t,n}^{\circ}\right| = \left|X_{t,n}'\right| + M_{t,n}'$. 

However, $\left|X_{t,n}\right| = \left|X_{t,n}^{\circ}\right| + \epsilon$, where $\epsilon$ is the number of conjugacy classes in $X_{t,n}$ such that their braid closure is a link of braid index one or two. Proposition~\ref{ptraceexponent} implies that $M_{t,n} = M_{t,n}' + \epsilon$ (Definition~\ref{dMnumbers}). Therefore, the statement is established.
\end{proof}

We also observe the following statement.

\begin{proposition}
\label{pMnumbers}
The number $M_{t,n} = 0$ if $n < -\left|t - 3\right|$.
\end{proposition}
\begin{proof}
The number $M_{t,n}'$ is the cardinality of the union of two sets in Definition~\ref{dMnumbers}. If $n<-\left|t-3\right|$, then it is straightforward to check that the union of the two sets is empty and $M_{t,n} = M_{t,n}'$.
\end{proof}

\section{The Alexander and Jones polynomials}
\label{sAlexanderJones}

The (reduced) \textit{Burau representation} $\beta_3:B_3\to \text{GL}_2\left(\mathbb{Z}\left[q^{\pm 1}\right]\right)$ of the braid group $B_3$ is defined by the rule \[\beta_3\left(\sigma_1\right) = \begin{bmatrix} 1 & -q \\ 0 & -q \end{bmatrix} \text{ and } \beta_3\left(\sigma_2\right) = \begin{bmatrix} -q & 0 \\ -1 & 1 \end{bmatrix}.\] We observe that the specialization of $\beta_3$ at $q = -1$ is the map $\phi:B_3\to \text{SL}_2\left(\mathbb{Z}\right)$ in Section~\ref{sconjugationB3SL2}. The other definitions of the Burau representation in the literature differ from ours by a change of basis. 

Let $L\subseteq \mathbb{S}^3$ be a link of braid index at most three, in which case $L$ is isotopic to the braid closure $\overline{g}$ for some $g\in B_3$. In this section, we recall the definitions of the Alexander polynomial of $L$ and the Jones polynomial of $L$ in terms of (the trace of) the Burau matrix $\beta_3\left(g\right)$. We refer the reader to~\cite{jones1987hecke} for further discussion. 

We denote the \textit{Alexander polynomial} of $L$ by $\Delta_{L}\left(q\right)\in \mathbb{Z}\left[\sqrt{q},\sqrt{q^{-1}}\right]$ and we denote the \textit{Jones polynomial} of $L$ by $V_{L}\left(q\right)\in \mathbb{Z}\left[\sqrt{q},\sqrt{q^{-1}}\right]$. If $g\in B_3$ and if $\overline{g}$ is the braid closure of $g$, then we have the following formulas, which can be derived from~\cite{jones1987hecke}: \begin{align*} \Delta_{\overline{g}}\left(q\right) &= \left(-\frac{1}{\sqrt{q}}\right)^{\epsilon\left(g\right)-2} \frac{1-\text{tr}\left(\beta_3\left(g\right)\right) + \left(-q\right)^{\epsilon\left(g\right)}}{1+q+q^2} \\ V_{\overline{g}}\left(q\right) &= \left(\sqrt{q}\right)^{\epsilon\left(g\right)}\left(q + q^{-1} + \text{tr}\left(\beta_3\left(g\right)\right)\right).\end{align*} In particular, both the Alexander polynomial and the Jones polynomial of $\overline{g}$ are determined by $\text{tr}\left(\beta_3\left(g\right)\right)$ and $\epsilon\left(g\right)$. We also have the following equation which can be derived from the previous two equations (a special case of the equation is stated in~\cite{jones1987hecke} when $L=\overline{g}$ is a knot (in which case $\epsilon\left(g\right)$ is even)): \[V_{\overline{g}}\left(q\right) = \sqrt{q}^{\epsilon\left(g\right)}\left[q+q^{-1} + 1 + \left(-q\right)^{\epsilon\left(g\right)} - \left(-\sqrt{q}\right)^{\epsilon\left(g\right)-2}\left(1+q+q^2\right)\Delta_{\overline{g}}\left(q\right)\right].\] The equation gives an interpretation of $\epsilon\left(g\right)$ in terms of known knot invariants, but the equation only determines $\epsilon\left(g\right)$ precisely if $L$ is a link of braid index equal to three. In this case, $\epsilon\left(g\right)$ is equal to the writhe $e\left(L\right)$, and $e\left(L\right)$ is determined by the Alexander polynomial and the Jones polynomial of $L$. (If $L$ is the unknot for example, then $\Delta_L = 1 = V_L$ and $\epsilon\left(g\right)\in \{-2,0,2\}$ is not uniquely determined by $L$.)

The skein relations (or the previous equation) show that $V_{L}\left(-1\right) = \Delta_{L}\left(-1\right)$. The previous equations show that \[V_L\left(-1\right) = i^{e}\left(\text{tr}\left(g\right)-2\right) = \Delta_L\left(-1\right).\] 

We can rewrite the definition of the sets $X_{t,n}'$ using either the Alexander polynomial or the Jones polynomial.

\begin{lemma}
\label{lAlexanderJonestraceexponent}
Let us fix $P_{L}$ to be either the Alexander polynomial of $L$ or the Jones polynomial of $L$. The set $X_{t,n}'$ is the set of all isotopy classes of links of braid index three such that the writhe $e\left(L\right) = n$ and the special value $P_L\left(-1\right) = i^{n}\left(t - 2\right)$. 
\end{lemma}
\begin{proof}
The statement is an immediate consequence of the definitions of the Alexander polynomial and Jones polynomial for the closure of a $3$-braid, recalled above.
\end{proof}

\section{Proof of the main result}
\label{sproof}

We state and prove the main result of this paper.

\begin{theorem}
\label{tmain}
Let $t$ and $n$ be integers with $t\neq \pm 2$. Let $h\left(t\right)$ denote the number of equivalence classes of integral binary quadratic forms of discriminant $t^2 - 4$. If $L$ is a link, then we fix $P_L$ to be either the Alexander polynomial of $L$ or the Jones polynomial of $L$. Let $p_{t,n}$ be as defined in the statement of Theorem~\ref{tmainconsequence} in the Introduction. Let $M_{t,n}$ be as in Definition~\ref{dMnumbers}. We have $h\left(t\right) = \sum_{j=0}^{11} \left(p_{t,n+j} + M_{t,n+j}\right)$.
\end{theorem}
\begin{proof}
If $n$ is an integer, then \begin{align*} h\left(t\right) &= \left|\coprod_{j=0}^{11} X_{t,n+j}\right| \\ &= \left|\coprod_{j=0}^{11} X_{t,n+j}'\right| + \sum_{j=0}^{11} M_{t,n+j} \\ &= \sum_{j=0}^{11} \left(p_{t,n+j} + M_{t,n+j}\right). \end{align*} The first equality is Lemma~\ref{lLatimer-MacDuffee}. The second equality is Lemma~\ref{lbraidlinkdiscrepancy}. The third equality is a consequence of Lemma~\ref{lAlexanderJonestraceexponent}. 
\end{proof}

The statements in the Introduction are corollaries of Theorem~\ref{tmain}.

\begin{proof}[Proof of Theorem~\ref{tmainconsequence}]
The statement is a consequence of Theorem~\ref{tmain} and Proposition~\ref{pMnumbers}.
\end{proof}

\begin{proof}[Proof of Theorem~\ref{tfiniteness}]
If $E_{P}$ is the set of possible writhes of links $L$ such that $P_L = P$, then $E_P$ is finite. (Indeed, this is a consequence of an analysis of the Burau representation $\beta_3$ and the definitions of the Alexander polynomial and the Jones polynomial in Section~\ref{sAlexanderJones}.) Theorem~\ref{tmain} now shows that the statement is equivalent to Lagrange's theorem on the finiteness of $h\left(t\right)$. 
\end{proof}

\begin{proof}[Proof of Theorem~\ref{tsymmetry}]
The statement is an immediate consequence of Theorem~\ref{tmain}, Lemma~\ref{lmainconnection}, Lemma~\ref{lbraidlinkdiscrepancy}, Proposition~\ref{pMnumbers}, and Lemma~\ref{lAlexanderJonestraceexponent}, since $h\left(t\right) = h\left(-t\right)$.
\end{proof}

\bibliography{References}
\bibliographystyle{plain}
\end{document}